\documentclass[11pt]{amsart}

\usepackage{amsmath}
\usepackage{graphicx}
\usepackage{amssymb}
\usepackage{epstopdf}

\newcommand{\Z}{\mathbb{Z}}  
\newcommand{\R}{\mathbb{R}}  

\newcommand{\eps}{\varepsilon}

\newcommand{\T}{\mathbb{T}}

\renewcommand{\hat}[1]{\widehat{#1}}

\theoremstyle{plain} 
\newtheorem{thm}{Theorem}[section]
\newtheorem{lem}{Lemma}[section]
\newtheorem{pro}{Proposition}[section]
\newtheorem{cor}{Corollary}[section]

\theoremstyle{definition}

\theoremstyle{remark}

\title{On Bohr sets of integer-valued traceless matrices}

\author{Alexander Fish}
\address{School of Mathematics and Statistics, University of Sydney, Australia}
\curraddr{}
\email{alexander.fish@sydney.edu.au}
\thanks{}



\begin{document}

\begin{abstract}
In this paper we show that any Bohr-zero non-periodic set $B$ of traceless integer-valued matrices, denoted by $\Lambda$, intersects non-trivially the conjugacy class of any matrix from $\Lambda$. As a corollary, we obtain that the family of  characteristic polynomials of $B$  contains all characteristic polynomials of matrices from $\Lambda$. The main ingredient used in this paper is an equidistribution  result for an $SL_d(\Z)$ random walk on a finite-dimensional torus deduced from  Bourgain-Furman-Lindenstrauss-Mozes work \cite{BFLM}.
\end{abstract}
\maketitle
\section{Introduction}

 Let us denote by $Mat_d^0(\mathbb{Z})$, $d \geq 2$,  the set of integer-valued $d\times d $ matrices with zero trace, and by $\T^n$, $n \geq 1$, the $n$-dimensional torus $\R^n / \Z^n$. Let $G$ be a countable abelian group.  A set $B \subset G$ is called a \textit{non-periodic Bohr set} if there exist a homomorphism $\tau:G \to \mathbb{T}^n$, for some $n \geq 1$, with $\overline{\tau(G)} = \mathbb{T}^n$, and an open set $U \subset \mathbb{T}^n$ satisfying $B = \tau^{-1}(U)$. If the open set $U$ contains the zero element of $\mathbb{T}^n$, then  the set $B$ is called a \textit{Bohr-zero set}.  We will also denote by $SL_d(\Z)$ the group of $d \times d$ integer-valued matrices of determinant one.

The main result of this paper is the following.
\medskip

\noindent \textbf{Main Theorem.} \textit{Let $d \geq 2$, and $B \subset Mat_d^0(\mathbb{Z})$ be a Bohr-zero non-periodic set. Then for any matrix $C \in Mat_d^0(\Z)$ there exists a matrix $A \in B$ and a matrix $g \in SL_d(\Z)$ such that $C = g^{-1}Ag$. }
\medskip

The same result has been also proved independently by Bj\"orklund and Bulinski \cite{BB}. They use the recent works of Benoist-Quint \cite{BQ2} and \cite{BQ3}, instead of the work of Bourgain-Furman-Lindenstrauss-Mozes as the main ingredient in the proof. Use of Bourgain-Furman-Lindenstrauss-Mozes work enables us to prove a strong equidistribution result for the random walk of $SL_d(\Z)$ acting on $Mat_d^0(\R) / Mat_d^0(\Z)$ by the conjugation (see Theorem \ref{equid2}). This result is interesting by its own, and may have other number-theoretic applications. This should be compared with Theorem 1.15 in \cite{BB}, where the equidistribution is established for Ces\`aro average of the random walk (a weaker statement).  

 \begin{cor}\label{cor1}
Let $d \geq 2$, and $B \subset Mat_d^0(\mathbb{\Z})$ be a Bohr-zero non-periodic set. The set of characteristic polynomials of the matrices in $B$ coincides with the set of all characteristic polynomials of the matrices in $Mat_d^0(\Z)$.
\end{cor}
\medskip

The following number-theoretic statement is an immediate implication of Corollary \ref{cor1}.
\begin{cor}\label{cor2}
Let $B \subset \Z$ be a Bohr-zero non-periodic set. Then the set of the discriminants over $B$ defined by 
\[
D := \{xy-z^2 \, | \, x,y,z \in B\}
\]
satisfies that $D = \Z$.
\end{cor}

At this point we will define Furstenberg's system corresponding to a set $B$ of positive density in a countable abelian group $G$. It is well known that in any (countable) abelian group we can find sequences of almost invariant finite sets, so called F\o lner sequences. A sequence of finite sets $(F_n)$ in $G$ is called \textit{F\o lner} if it is asymptotically $G$-invariant, i.e. for every $g \in G$ we have $\frac{\left| F_n \cap \left( F_n + g \right) \right|}{|F_n|} \to 1, \mbox{ as } n \to \infty$. 
We will say that $B$ has \textit{positive density} if  \textit{upper Banach density} of $B$ is positive, i.e., if
\[
d^*(B) = \sup_{(F_n) \subset G \mbox{ F\o lner}} \limsup_{n \to \infty} \frac{|B \cap F_n|}{|F_n|}
\] 
is positive. Furstenberg in his seminal paper \cite{Fu} constructed a $G$-measure-preserving system\footnote{A triple $(X,\eta,\sigma)$ is a $G$-\textit{measure-preserving system}, if $X$ is a compact metric space on which acts $G$ by a measurable action denoted by $\sigma$,  $\eta$ is a Borel probability measure on $X$, and the action of $G$ preserves $\eta$.}
$(X,\eta,\sigma)$ and a clopen set $\tilde{B} \subset X$ such that
\medskip

\begin{itemize}
\item $d^*(B \cap (B+h)) \geq \eta \left(\tilde{B} \cap \sigma(h)\tilde{B}\right), \mbox{ for any } h \in G$.\\
\item $\eta(\tilde{B}) = d^*(B)$.
\end{itemize} 
\medskip 

\noindent Moreover, we can further refine the statement of the correspondence principle\footnote{The proof of Correspondence Principle II of Appendix I in \cite{BF} will work also for the finite intersections, and therefore, will imply the claim.}, and require from the system to be ergodic\footnote{A $G$-measure-preserving system is \textit{ergodic} if any $G$-invariant measurable set has measure either zero or one.}. 
We will denote Furstenberg's (ergodic) system corresponding to $B$ by $X_{B} = (X,\eta,\sigma, \tilde{B})$.
Next, we will define  the notion of the spectral measure corresponding to a set $B$ of a countable abelian group $G$ of positive density and its Furstenberg's system $X_B = (X,\eta,\sigma,\tilde{B})$. 
Denote by $1_{\tilde{B}}$ the indicator function of the set $\tilde{B}$. Then by Bochner's spectral theorem \cite{Fo} there exists  a non-negative finite Borel measure $\nu$ on $\hat{G}$ (the dual of $G$) which satisfies:
\[
\langle 1_{\tilde{B}}, \sigma(h) 1_{\tilde{B}}\rangle = \int_{\hat{G}} \chi(h) d \nu(\chi), \mbox{ for } h \in G.
\]
The measure $\nu$ will be called the \textit{spectral measure of the set $B$ and its Furstenberg's system $X_B$}, and we will denote by $\hat{\nu}(h)$ the right hand side of the last equation. 
We are at the position to state the main technical claim of the paper (see Section \ref{sec3} for the proof).
\medskip

\begin{thm}\label{techn_thm}
Let $d \geq 2$, and let $B \subset Mat_d^0(\Z)$ be a set of positive density such that  the spectral measure\footnote{We assume the existence of some Furstenberg's system $X_B$ corresponding to the set $B$, such that the associated spectral measure satisfies the requirement of the theorem.} of $B$ has no atoms at non-trivial characters having finite torsion. Then for every $C \in Mat_d^0(\Z)$ there exist $A \in B-B$ and $g \in SL_d(\Z)$ with $C = g^{-1}Ag$. 
\end{thm}
\medskip

Theorem \ref{techn_thm} is an extension  of the following result that has been proved in \cite{BF} by use of the equidistribution result of Benoist-Quint \cite{BQ}.
\medskip

\begin{thm}
\label{BF}
Let $d \geq 2$, and let  $B \subset Mat_d^0(\Z)$ be a set of positive density. Then there exists $k \geq 1$ such that for any matrix $C \in k Mat_d^0(\Z)$ there exists $A \in B - B$ and $g \in SL_d(\Z)$ with 
$C = g^{-1} A g$. 
\end{thm}
\medskip

We would like to finish the introduction by stating the piecewise version of Main Theorem. We recall that a set $B \subset G$ is called \textit{piecewise (non-periodic) Bohr set} if there is a (non-periodic) Bohr set $B_0 \subset G$ and a (thick) set $T \subset G$ of upper Banach density one, i.e., $d^*(T) = 1$ such that $B = B_0 \cap T$. 
Theorem \ref{techn_thm} implies the following result (see the proof in Section \ref{sec3}).
\medskip

\begin{thm}\label{PWBohrthm}
Let $d \geq 2$, and let $B \subset Mat_d^0(\Z)$ be a piecewise Bohr non-periodic set. Then for every $C \in Mat_d^0(\Z)$ there exist $A \in B-B$ and $g \in SL_d(\Z)$ with $C = g^{-1}Ag$. 
\end{thm}
\medskip

Let us show that Theorem \ref{PWBohrthm} implies Main Theorem.
\medskip

\noindent \textbf{Proof of Main Theorem.} Let  $B \subset Mat_d^0(\Z)$ be a Bohr-zero non-periodic set. Notice that there exists $B_0 \subset Mat_d^0(\Z)$ a Bohr-zero non-periodic set with the property that 
\[
B_0 - B_0 \subset B.
\]
Now, we apply Theorem \ref{PWBohrthm} for the set $B_0$, and as a conclusion obtain the statement of the theorem.
\qed
\medskip

\noindent \textit{Organisation of the paper.} In Section \ref{sec2} we establish the consequences of the equidistribution result of Bourgain-Furman-Lindenstrauss-Mozes \cite{BFLM} related to the adjoint action of $SL_d(\Z)$ on $Mat_d^0(\R) / Mat_d^0(\Z)$. In Section \ref{sec3} we prove Theorems \ref{techn_thm}, and \ref{PWBohrthm}. 
\medskip

\noindent \textit{Acknowledgement.} The author is grateful to Ben Green and Tom Sanders for fruitful discussions on the topic of the paper. He is also grateful to Benoist Quint for explaining certain ingredients of \cite{BQ3}, and to Shahar Mozes for insightful discussions related to the paper \cite{BFLM}. And, finally, the author would like to thank an anonymous referee for valuable suggestions improving the presentation in the paper.
\medskip

\section{Consequences of the work of Bourgain-Furman--Lindenstrauss-Mozes}
\label{sec2}

We start by recalling the property of strong irreducibility of an action of a discrete group. Let $\Gamma$ be a countable group, and $V$ be a finite dimensional real space. We say that an action $\rho:\Gamma \to End(V)$ is \textit{strongly irreducible} if for every finite index subgroup $H$ of  $\Gamma$, the restriction of the action of $\rho$ to $H$ is irreducible. We also will be using the notion of a proximal element. An operator $T \in End(V) $ will be called \textit{proximal}, if there is only one eigenvalue of the largest absolute value, and corresponding to it eigenspace is one-dimensional.

Assume that a countable group $\Gamma$ acts on a compact Borel measure space $(X,\nu)$. Let $\mu$ be a probability measure on $\Gamma$. Then the convolution measure $\mu \ast \nu$ on $X$ is defined by:
\[
\int_X f d\left( \mu \ast \nu\right) = \int_X \left( \sum_{g \in \Gamma} f(gx)\mu(g) \right) d \nu(x), \mbox{ for any } f \in C(X).
\]
We will denote the Dirac probability measure at a point $x \in X$ by $\delta_x$. For every $k \geq 2$, we define the probability measure $\mu^{*k}$ on $\Gamma$ by 
\[
\mu^{*k}(g) = \sum_{g_1 \cdot \ldots  \cdot g_k =g} \mu(g_1)\mu(g_2)\ldots \mu(g_k).
\]

The main ingredient in the proofs of all our main results is the following seminal equidistribution statement due to Bourgain-Furman-Lindenstrauss-Mozes \cite{BFLM}.

\begin{thm}[Corollary B in \cite{BFLM}]\label{BFLM}
Let $\Gamma < SL_n(\Z)$ be a subgroup which acts totally irreducibly on $\R^n$, and having a proximal element. Let $\mu$ be a finite generating probability measure on $\Gamma$. Let $x \in \T^n$ be a non-rational point. Then the measures $\mu^{*k} \ast \delta_x$ converge in weak$^*$-topology as $k \to \infty$ to Haar measure on $\T^n$.
\end{thm}
   
In this note, the acting group will be $\Gamma = SL_d(\Z)$. The group $\Gamma$ acts by the conjugation on the real vector space $V = Mat_d^0(\R)$ of real valued $d \times d$ matrices with zero trace. So, an element  $g \in SL_d(\Z)$ acts on $v \in V$ by $Ad(g)v = g^{-1}vg$, and such action called the \textit{adjoint action} of $SL_d(\Z)$.  Notice that $V$ is  isomorphic to $\R^{d^2-1}$ and for any element $g$ of $SL_d(\Z)$, the endomorphism  of $V$ obtained by the conjugate action of $g$ has determinant one. The next claim will allow us to apply Theorem \ref{BFLM} in our setting.

\begin{pro}\label{action}
The adjoint action of  $SL_d(\mathbb{Z})$ on $Mat_d^0(\mathbb{R})$ is strongly irreducible, and  $SL_d(\Z)$ contains an element which acts proximally.
\end{pro}
\begin{proof}
It is proved in \cite{BF} [Corollary 5.4] that the adjoint action of $SL_{d}(\Z)$ on $Mat_{d}^0(\R)$ is strongly irreducible. Therefore,  it is remained to prove that there is at least one element of $SL_d(\Z)$ which acts on $Mat_{d}^0(\R)$ proximaly. Proposition \ref{sub_pro} below finishes the proof of the statement. 
\end{proof}

\begin{pro}\label{sub_pro}
The matrix 
\[
B_2 = \left[
\begin{array}{cc}
    1       & -1 \\
   -1 & 2\\
\end{array}
\right]
\]
acts (by conjugation) on $Mat_2^0(\R)$ proximally. 
For $d \geq 3$, the matrix
\[
B_d = \left[
\begin{array}{cc|c}
1 & -1& \mathbf{0}_{2 \times (d-2)}\\
-1 & 2 &\\
\hline
\mathbf{0}_{(d-2) \times (d-2)}& &  Id_{(d-2)\times (d-2)}
\\
\end{array}
\right]
\]
acts proximally on $Mat_d^0(\R)$. 
\end{pro}
\begin{proof}
It is straithforward  to check that the  operator $B_2: Mat_{2}^0(\R) \to Mat_{2}^0(\R)$ can be written in the matrix form as\footnote{We use the identification between $Mat_{2}^0(\R)$ and $\R^3$, by 
\[
\left[
\begin{array}{cc}
    x      &y \\
   z & -x
\end{array}
\right] 
\to 
[ x,y,z ]^t
\]
}
\[
\left[
\begin{array}{rrr}
    3       & -2 & 1 \\
  -4 & 4 & -1\\
  2 & -1 & 1
\end{array}
\right].
\]
The characteristic polynomial  of this operator is $\chi_{B_2}(\lambda) = (1- \lambda) (\lambda^2 - 7\lambda + 1)$. Since all eigenvalues are distinct by their absolute value, it follows that the operator acts proximally.
\medskip

In the case $d  \geq 3$, notice that the action of $B_d$ on $Mat_d^0(\R)$ is decomposed into $4$ orthogonal spaces. The actions on the $2 \times 2$ upper left corner, $2 \times (d-2)$ upper right corner, $(d-2) \times 2$ bottom left corner, and the identity action on the bottom right $(d-2) \times (d-2)$ corner. Correspondingly, the dimensions of the spaces are $4, 2 \cdot (d-2), (d-2) \cdot 2,$ and $(d-2)^2 - 1$\footnote{We identify the vector space $Mat_d^0(\R)$ with $\R^{d^2-1}$ by omitting the $(d,d)$-entry of matrices in $Mat_d^0(\R)$}.  
\medskip

The $4$-dimensional left upper corner part can be written in the matrix form as\footnote{We choose the standard basis for $Mat_2(\R)$: 
\[
\left[
\begin{array}{cc}
    1      &0 \\
   0 & 0
\end{array}
\right], 
 \left[
\begin{array}{cc}
    0      &1 \\
   0 & 0
\end{array}
\right],
\left[
\begin{array}{cc}
    0     &0 \\
   1 & 0
\end{array}
\right],
\left[
\begin{array}{cc}
    0      &0 \\
   0 & 1
\end{array}
\right].
\] }

\[
\left[
\begin{array}{rrrr}
    2       & -2 & 1 & -1 \\
  -2 & 4 & -1 & 2\\
  1 & -1 & 1 & -1\\
  -1 & 2 & -1 &2
\end{array}
\right].
\]
Its characteristic polynomial is $(\lambda - 1)^2(\lambda^2 - 7 \lambda + 1)$. Therefore there is a unique highest eigenvalue by the absolute value equal to $\frac{7 + 3\sqrt{5}}{2}$, and it has multiplicity one. 
\medskip

The operator $B_d$ acts on the upper right corner  in the following way
\[
\left[
\begin{array}{cccc}
    x_1  &y_1 \\
    x_2 & y_2 \\
    \ldots & \ldots \\
    x_{d-2} & y_{d-2}
\end{array}
\right]^t \to 
\left[
\begin{array}{cc}
    2x_1 + y_1    &  x_1 + y_1 \\
    2x_2 + y_2 & x_2 + y_2 \\
     \ldots & \ldots \\
    2x_{d-2} + y_{d-2} &x_{d-2} + y_{d-2}
\end{array}
\right]^t. 
\]
It is clear that it has two eigenvalues with multiplicity $d-2$. These eigenvalues correspond to the eigenvalues of the matrix 
\[
 C = \left[
\begin{array}{rr}
    2       & 1 \\
   1 & 1
\end{array}
\right].
\]
These eigenvalues are the roots of the characteristic polynomial of the matrix $C$ which are $\frac{3 \pm \sqrt{5}}{2}$. 
\medskip

The operator $B_d$ acts on the bottom left corner in the following way:
\[
\left[
\begin{array}{cccc}
    x_1  &y_1 \\
    x_2 & y_2 \\
    \ldots & \ldots \\
    x_{d-2} & y_{d-2}
\end{array}
\right] \to 
\left[
\begin{array}{cc}
    x_1 - y_1    &  -x_1 + 2y_1 \\
    x_2 - y_2 & -x_2 + 2y_2 \\
     \ldots & \ldots \\
    x_{d-2} - y_{d-2} & -x_{d-2} + 2y_{d-2}
\end{array}
\right]. 
\]
Therefore it has two eigenvalues of the matrix $C^{-1}$ each one having multiplicity $d-2$. It is immediate to check that $C^{-1}$ has the same characteristic polynomial as $C$, therefore the eigenvalues of the operator $B_d$ acting on the bottom left corner are $\frac{3 \pm \sqrt{5}}{2}$, each one having multiplicity $d-2$.
 \medskip
 
 As the conclusion of the previous considerations we find the the eigenvalues of the operator $B_d$ are 
 $\frac{7 + 3\sqrt{5}}{2},\frac{3 + \sqrt{5}}{2}, 1, \frac{7 - 3\sqrt{5}}{2}, \frac{3 - \sqrt{5}}{2}$  with corresponding algebraic multiplicities equal to $1, 2(d-2), [(d-2)^2 -1] + 2, 1, 2(d-2)$. This implies that $B_d$ acts proximally on $Mat_d^0(\R)$.
\end{proof}

Let us denote by $A_d = V / \Lambda$, where $\Lambda = Mat_d^0(\Z)$. Notice that $A_d$ is isomorphic to $\T^{d^2-1}$, and it is the dual group of $\Lambda$. The adjoint action of $SL_d(\Z)$ leaves $\Lambda$ invariant. Therefore, $SL_d(\Z)$ also acts on $A_d$. Proposition \ref{action} implies by Theorem \ref{BFLM} the following statement.

\begin{pro}\label{equidistribution}
Let $\mu$ be a probability measure on $SL_d(\Z)$ with finite generating support. Let $x \in A_d$ be a non-rational point. Then the measures $\mu^{*k} \ast \delta_x$ converge as $k \to \infty$ in the weak$^*$ topology to the normalised Haar measure on $A_d$.
\end{pro}

We will be using Proposition \ref{equidistribution} to prove the following:

\begin{thm}\label{equid2}
Let $\mu$ be a probability measure on $SL_d(\Z)$ with finite generating support. Let $\nu$ be a probability measure on $A_d$ with no atoms at rational points. Then the measures $\mu^{*k} \ast \nu$ converge as $k \to \infty$ in the weak$^*$ topology to the normalised Haar measure on $A_d$.
\end{thm}
\begin{proof}
Let $\nu$ be a probability measure on $A_d$ with no atoms at rational points, and let $\mu$ be a probability measure on $\Gamma = SL_d(\Z)$ with a finite generating support. By Proposition \ref{equidistribution}  for every non-rational $x \in A_d$ the measures $\mu^{*k} \ast \delta_x$ converge in weak$^*$-topology as $k \to \infty$ to the Haar measure on $A_d$. Let $f $ be a continuous function on $A_d$. Denote by $A_d'$ the set of all non-rational points of $A_d$. Notice that $\nu(A_d') = 1$.
Then for every $x \in A_d'$ we have that $f_k(x) := \int f d \left(\mu^{*k} \ast \delta_x \right)\to \int f$. We have to show that 
\[
\int_{A_d} f d \left(  \mu^{*k} \ast \nu \right) \to \int f. 
\]
By Egorov's theorem, for every $\eps > 0$, there exists $X' \subset A_d'$ with $\nu(X') \geq 1-\eps$ and $K(\eps)$ with the property that for every $x \in X'$ and every $k \geq K(\eps)$ we have
\[
\left| f_k(x) - \int f \right|< \eps.
\]
Notice that 
\[
\int f d \left( \mu^{*k} \ast \nu \right)= \sum_{g \in \Gamma} \int f(gx) \mu^{*k}(g) d \nu(x) = \int \left( \sum_{g \in \Gamma} f(gx) \mu^{*k}(g)\right) d \nu(x)
\]
\[
= \int \left( \int f d \left( \mu^{*k} \ast \delta_x \right) \right) d \nu(x) = \int f_k(x) d \nu(x). 
\]
Let $\delta > 0$. Denote by $M = \| f \|_{\infty}$, and take $\eps > 0$ so small that $\eps M < 2\delta$, and $\eps < \delta$. Then we have
\[
\left| \int f d \left( \mu^{*k} \ast \nu \right) - \int f \right|  = \left| \int \left( f_k(x) - \int f \right) d \nu(x)\right| < (1 - \eps)\eps + \eps M < 3 \delta,
\]
for $k \geq K(\eps)$. Since $\delta$ can be chosen arbitrary small, we have shown that 
\[
\int f d \left( \mu^{*k} \ast \nu \right) \to \int f.
\]
This finishes the proof because the function $f$ was an arbitrary continuous function on $A_d$.
\end{proof}

\section{Proofs of Theorems \ref{techn_thm}, and \ref{PWBohrthm}}
\label{sec3}
\medskip

First, we will prove a very useful statement.

\begin{lem}
\label{character_lem}
Let $G$ be a countable abelian group, and let $(F_n) \subset G$ be a F\o lner sequence. Let $\chi \in \hat{G}$ be a non-trivial character. Then we have
\[
\frac{1}{|F_n|} \sum_{g \in F_n} \chi(g) \to 0, \mbox{ as } n \to \infty.
\]
\end{lem}
\begin{proof}
Let $(n_k)$ be a sequence along which the limit of 
\[
\frac{1}{|F_{n_k}|} \sum_{g \in F_{n_k}} \chi(g) 
\]
exists. Let us denote the limit by $F(\chi)$. For every $h \in G$ we have:
\[
\frac{1}{|F_{n_k}|} \sum_{g \in F_{n_k}} \chi(g+h) = \chi(h) \frac{1}{|F_{n_k}|} \sum_{g \in F_{n_k}} \chi(g) \to \chi(h) F(\chi).
\] 
On the other hand,  by use of F\o lner property of $(F_{n_k})$ we obtain:
\[
\frac{1}{|F_{n_k}|} \sum_{g \in F_{n_k}} \chi(g+h) = \frac{1}{|F_{n_k}|} \sum_{g \in F_{n_k} + h} \chi(g) \to F(\chi). 
\]
Therefore, we have for every $h \in G$:
\[
F(\chi) = \chi(h) F(\chi).
\]
Since $\chi \not \equiv 1$, we get that $F(\chi) = 0$. The statement of the lemma follows, since our conclusion is independent of the subsequence $(n_k)$. 
\end{proof}

Recall that $\Lambda = Mat_d^0(\Z)$, and we denote by $\Gamma = SL_d(\Z)$. We make the identification of the dual space of $\Lambda$ with the torus $A_d = Mat_d^0(\R) / Mat_d^0(\Z) \simeq \T^{d^2-1}$ by corresponding for every $x \in A_d$  the character $\chi_x$ on $\Lambda$ given by:
\[
\chi_x(h) = \exp{(2 \pi i\langle x, h \rangle)}, \mbox{ for } h \in \Lambda.
\]
Notice that the trivial character on $\Lambda$ corresponds to the zero element $o_{A_d}$ of $A_d$, and characters having finite torsion correspond to the rational points of $A_d$. 
 \medskip
 
\noindent \textbf{Proof of Theorem \ref{techn_thm}.}
Let $B \subset \Lambda$ be a set of positive density with Furstenberg's system $X_B = (X,\eta,\sigma,\tilde{B})$ and such that the spectral measure of $B$ has no atoms at non-trivial characters. 
Denote by $\nu$ the spectral measure of $B$, i.e., for every $h \in \Lambda$ we have
\begin{equation}\label{Fourier}
\langle 1_{\tilde{B}}, \sigma(h) 1_{\tilde{B}} \rangle = \int_{A_d}  \exp{\left(2 \pi i \langle x, h \rangle\right)} d \nu(x).
\end{equation}
By the assumptions of the theorem, $\nu$ has no atoms at the rational points.  We will show that for every $h \in \Lambda$ there exists $g \in \Gamma$ such that 
\[
\hat{\nu}(g^{-1}hg) = \langle 1_{\tilde{B}}, \sigma(g^{-1}hg) 1_{\tilde{B}} \rangle  > 0.
\]
This will imply the claim of the theorem by the first property of Furstenberg's system $X_B$. Assume, that on the contrary, there exists $h \in \Lambda$ such that for all $g \in \Gamma$ we have 
\begin{equation}\label{contr}
\hat{\nu}(g^{-1}hg) = 0.
\end{equation}
Since for $h = 0_{\Lambda}$ we have $g^{-1} 0_{\Lambda} g = 0_{\Lambda}$, and $\hat{\nu}(0_{\Lambda}) = \eta(\tilde{B}) > 0$, we conclude that there exists a non-zero $h \in \Lambda$ such that $(\ref{contr})$ holds for all $g \in \Gamma$.
%

\noindent For any F\o lner sequence $(F_n)$ in $\Lambda$:
\begin{equation}\label{erg}
\frac{1}{|F_n|} \sum_{h \in F_n} \langle 1_{\tilde{B}}, \sigma(h) 1_{\tilde{B}} \rangle =
\end{equation}
\[
\int_{A_d} \frac{1}{|F_n|} \sum_{h \in F_n} \exp{(2 \pi i \langle x, h \rangle)} d \nu(x) \to \nu(\{o_{A_d}\}), \mbox{ as } N \to \infty.
\]
In the last transition, we have used Lebesgue's dominated convergence theorem and Lemma \ref{character_lem}.
By ergodicity of Furstenberg's system and von-Neumann's ergodic theorem it follows that the left hand side of (\ref{erg}) satisfies
\[
\frac{1}{|F_n|} \sum_{h \in F_n} \langle 1_{\tilde{B}}, \sigma(h) 1_{\tilde{B}} \rangle \to \eta(\tilde{B})^2, \mbox{ as } n \to \infty. 
\] 
Altogether it implies that
\[
\nu(\{o_{A_d}\}) = \eta(\tilde{B})^2 > 0.
\]
\smallskip

Let $\mu$ be a probability measure on $\Gamma$ having a finite generating support. By Proposition \ref{equid2} the measures $\mu^{*k} \ast \nu$ converge as $k \to \infty$ in weak$^*$-topology to 
\[ \eta(\tilde{B})\left(1 - \eta(\tilde{B})\right) m_{A_d} + \eta(\tilde{B})^2 \delta_{o_{A_d}}, \] where $m_{A_d}$ stands for the normalised Haar measure on $A_d$. Notice that $\Gamma$ also acts on $A_d$ by $g \cdot x = (g^{t})^{-1} x g^t$, for $g \in \Gamma$. The action of $\Gamma$ on $A_d$ and the adjoint action of $\Gamma$ on $\Lambda$ are related by the following: 
\[
\langle (g \cdot x), h \rangle = \langle x , Ad(g)h \rangle, \mbox{ for every } g \in \Gamma, h \in \Lambda, x \in A_d.
\]
Notice 
\[
\hat{\mu^{*k} \ast \nu}(h) = \int_{A_d} \exp{(2 \pi i \langle x, h \rangle)} d\left( \mu^{*k} \ast \nu \right)(x) =
\]
\[
 \int_{A_d} \left( \sum_{g \in \Gamma}  \exp{(2 \pi i \langle (g \cdot x), h \rangle)} \mu^{*k}(g) \right)  d \nu(x)=
\]
\[
 \sum_{g \in \Gamma} \left( \int_{A_d}   \exp{(2 \pi i \langle x , \left( g^{-1} h g \rangle\right))}  d \nu(x) \right)  \mu^{*k}(g) = 
\sum_{g \in \Gamma}  \hat{\nu}(g^{-1}hg) \mu^{*k}(g).
\]
Recall,  we assumed that there exists a non-zero $h \in \Lambda$ such that $\hat{\nu}(g^{-1}hg) = 0$, for all $g \in \Gamma$. Therefore, we have $\hat{\mu^{*k} \ast \nu}(h) = 0$, for all $k \geq 1$. On other hand, since $\hat{m_{A_d}}(h) = 0$, and $\hat{\delta_{o_{A_d}}}(h) = 1$, we have:
\[
\hat{\mu^{*k} \ast \nu}(h) \to  \eta(\tilde{B})^2 > 0, \mbox{ as } k \to \infty.
\]
Thus, we have a contradiction. This finishes the proof of the theorem.

\qed
\bigskip

\noindent \textbf{Proof of Theorem \ref{PWBohrthm}.}
We will use the following statement which will be proved below.
\medskip

\begin{pro}\label{aperiodicity}
Let $B \subset \Lambda$ be a non-periodic piecewise Bohr set corresponding to a Jordan measurable\footnote{A set $A$ in a topological space $X$ equipped with a measure $m_X$ is \textit{Jordan measurable} if $m_X( \partial A) = 0$, where $\partial{A} = \overline{A} \setminus \overset{\circ}{A}$.} open set in a finite-dimensional torus. Then there exists a spectral measure  associated with $B$ that does not have atoms at non-zero rational points of $A_d$.
\end{pro}
\medskip

Let $B \subset \Lambda$ be a piecewise non-periodic Bohr set given by $B = \tau^{-1}(U) \cap T$, where $\tau:\Lambda \to \T^n$ is a homomorphism with a dense image, $U \subset \T^n$ is an open set, and $T \subset \Lambda$ is a set with $d^*(T) = 1$. Then $U$ contains an open ball $U_o$, and $m_{\T^n}(\partial U_o) = 0$, where $m_{\T^n}$ denotes the Haar normalised measure on $\T^n$. Denote by $B' = \tau^{-1}(U_o) \cap T \subset B$. The statement of Theorem \ref{PWBohrthm} for the non-periodic piecewise Bohr set $B'$ follows from Proposition \ref{aperiodicity} and Theorem \ref{techn_thm}. The latter implies the statement of the theorem for the set $B$.
\qed

%
%

\bigskip

\noindent \textbf{Proof of Proposition \ref{aperiodicity}.}
We are given a piecewise Bohr non-periodic set $B \subset \Lambda$ corresponding to a Jordan measurable open set in a finite dimensional torus. This means that $B = B_o \cap T$, where  $T \subset \Lambda$ with $d^*(T) = 1$, and $B_o = \tau^{-1}(U_o) \subset \Lambda$, where  
$\tau:\Lambda \to \T^n$, for some $n \geq 1$, is a homomorphism with a dense image, and  $U_o \subset \T^n$ is an open Jordan measurable set.
We will construct an ergodic Furstenberg's $\Lambda$-system $X_B = (X,\eta,\sigma,\tilde{B})$ corresponding to the set $B$, and will show that the spectral measure of the function $1_{\tilde{B}}$ has no atoms at the rational non-zero points of $A_d : = \hat{\Lambda}$.

Let $X = \T^n$, $\eta$ be the Haar normalised measure on $X$, $\sigma_h(x) := x + \tau(h)$ for $x \in X, h \in \Lambda$, and $\tilde{B} = U_o$. We will denote by $X_B := (X,\eta,\sigma, \tilde{B})$. It remains to show that 
\medskip

\begin{itemize}
\item For every $h \in \Lambda$ we have $d^*\left(B \cap (B +h)\right) \geq \eta(\tilde{B} \cap \sigma_h(\tilde{B}))$.\\
\item $\eta(\tilde{B}) = d^*(B)$.\\
\item The spectral measure of $1_{\tilde{B}}$ has no atoms at non-zero rational points of $A_d$.
\end{itemize}
\medskip
The first two properties will follow from the statement that for every $h \in \Lambda$:
\[
d^*\left(B \cap (B +h)\right) = \eta(\tilde{B} \cap \sigma_h(\tilde{B})).
\]
First, notice that for every $h \in \Lambda$ the set $U_o \cap \sigma_h(U_o)$ is Jordan measurable. The uniqueness of $\sigma$-invariant probability measure on $X$ implies the unique ergodicity of $X_B$. Therefore, for every F\o lner sequence $(F_k)$ in $\Lambda$ and any $h \in \Lambda$ we have
\[
\frac{1}{|F_k|} \sum_{g \in F_k} 1_{U_0 \cap \sigma_h(U_0)}(\sigma_g(0_X)) \to
\int_X 1_{U_o \cap \sigma_h(U_o)}(x) d\eta(x) =  \eta\left( \tilde{B} \cap \sigma_h(\tilde{B})\right),
\]
as $k \to \infty$. Since the left hand side of the last equation is equal to $\frac{|B_o \cap (B_o +h) \cap F_k|}{|F_k|}$, we obtain
\begin{equation}
\frac{|B_o \cap (B_o +h) \cap F_k|}{|F_k|} \to   \eta\left( \tilde{B} \cap \sigma_h(\tilde{B})\right), \mbox{ as } k \to \infty.
\label{limit}
\end{equation}
Since $B \subset B_0$, the latter implies that for every $h \in \Lambda$ we have
\begin{equation}
\label{ineq_upper}
 \eta\left( \tilde{B} \cap \sigma_h(\tilde{B})\right) \geq d^*(B \cap (B +h)).
\end{equation}
On the other hand, for any F\o lner sequence $(F_k)$ which lies inside the thick set $T$ by identity ($\ref{limit}$) we have for every $ h \in \Lambda$:
\[
\frac{|B \cap (B_0 +h) \cap F_k|}{|F_k|} \to  \eta\left( \tilde{B} \cap \sigma_h(\tilde{B})\right), \mbox{ as } k \to \infty.
\]
By use of F\o lner property of the sequence $(F_k)$, we have that 
\[
\left| \frac{|B \cap (B_0 +h) \cap F_k|}{|F_k|} - \frac{|(B-h) \cap B_0  \cap F_k|}{|F_k|} \right| \to 0, \mbox{ as } k \to \infty.
\]
But, since $F_k \subset T$ it follows that for every $k \geq 1$ we have: 
\[
\frac{|(B-h) \cap B_0  \cap F_k|}{|F_k|} = \frac{|(B-h) \cap B  \cap F_k|}{|F_k|}.
\]
Finally, since 
\[
\frac{|(B-h) \cap B  \cap F_k|}{|F_k|} = \frac{|B \cap (B+h)  \cap (F_k+h)|}{|F_k|},
\]
and F\o lner property implies that 
\[
\left| \frac{|B \cap (B+h)  \cap (F_k+h)|}{|F_k|} - \frac{|B \cap (B+h)  \cap F_k|}{|F_k|}\right|, \mbox{ as } k \to \infty,
\]
we obtain that
\[
\frac{|B \cap (B +h) \cap F_k|}{|F_k|} \to  \eta\left( \tilde{B} \cap \sigma_h(\tilde{B})\right), \mbox{ as } k \to \infty.
\]
This establishes that for every $h \in \Lambda$:
\[
d^*\left(B \cap (B +h)\right) \geq \eta(\tilde{B} \cap \sigma_h(\tilde{B})).
\]
Together with the idenity (\ref{ineq_upper}) this implies that for every $h \in \Lambda$ we have
\[
d^*\left(B \cap (B +h)\right) = \eta(\tilde{B} \cap \sigma_h(\tilde{B})).
\]
It remains to prove that the spectral measure corresponding to $1_{\tilde{B}}$ and the system $X_B$ has no atoms at non-zero rational points of $A_d$.
We will be abusing the notation and will also use $T$ to denote the Koopman operator on $L^2(X)$ corresponding to $\sigma$. 
Let us list two important properties of the system  $X_B$:
\medskip

\begin{enumerate}
\item \label{one} $X_B$ is \textit{totally ergodic}, i.e., every subgroup $H < \Lambda$ of a finite index acts ergodically on $X_B$.\\
\item\label{three} For every $ f \in L^2(X)$ there exists the spectral measure $\mu_f$ of $f$ on $A_d$ satisfying: 
\[
\hat{\mu_f}(h) := \int_{A_d} \exp{(2 \pi h \cdot x)} d \mu_f(x) = \langle f, T_h f \rangle. 
\] 
Moreover, if $f \geq 0$, then $\mu_f$  is non-negative.
\end{enumerate}
\medskip

The first property follows from Lemma \ref{lem2}, while the second property is Bochner's spectral theorem, see \cite{Fo}. To prove Lemma \ref{lem2} we will need the following result.

\begin{lem}\label{lem1}
Let $H < \Lambda$ be a subgroup of a finite index. Then for every point $x \in X$, the $H$-orbit of $x$, i.e., $\{\sigma_h(x) \, | \, h \in H\}$, is dense in $X$.
\end{lem}
\begin{proof}
If $\overline{\tau(H)} \neq X$, then since $H < \Lambda$ has a finite index, it follows that finitely many translates of $\overline{\tau(H)}$ cover $X$. But $X$ is connected, and we get a contradiction.

%
\end{proof}

\begin{lem}\label{lem2}
Let $H < \Lambda$ be a subgroup of a finite index. The restriction of the $\Lambda$-action of $X$ to $H$ is uniquely ergodic.
 \end{lem}
 \begin{proof}
 It follows from Lemma \ref{lem1} that any  $H$-invariant Borel probability measure on $X$ is also $X$-invariant. The uniqueness of the Haar normalised measure on $X$ implies the statement of the lemma. 
 \end{proof}
\bigskip

Let $f \in L^2(X)$, then by the ergodicity of $X_B$ (property (\ref{one})) it follows that for any F\o lner sequence $(F_k)_{k \geq 1}$ of finite sets in $\Lambda$ we have
\[
\frac{1}{|F_k|} \sum_{h \in F_k} \langle f, T_h f \rangle \to \left| \langle f,1\rangle\right|^2, \mbox{  as  } k \to \infty.
\]
On the other hand, it follows from Bochner's spectral theorem, Lebesgue's dominated convergence theorem and and Lemma \ref{character_lem} that
\[
\frac{1}{|F_k|} \sum_{h \in F_k} \langle f, T_{h} f \rangle \to \mu_f(\{o_{A_d}\}),
\]
which implies that 
\begin{equation}\label{eq-zero}
\left| \langle f, 1 \rangle \right|^2 = \mu_f(\{o_{A_d}\}).
\end{equation}
Let $x_0 \in A_d$ be a non-zero rational point with the least common denominator equal to $q$. Then the stabiliser of $x_0$ in $\Lambda$ is $H_{x_0} = q \Lambda$. Using the ergodicity of $H_{x_0}$ action on $X_B$ (property (\ref{one})), we obtain
\[
\frac{1}{|F_k|} \sum_{h \in F_k} \langle f, T_{qh} f \rangle \to \left| \langle f,1\rangle\right|^2, \mbox{  as  } k \to \infty.
\]
On the other hand, we have
 \[
\frac{1}{|F_k|} \sum_{h \in F_k} \exp{(2 \pi i \langle h, qx \rangle)} \to \left[ \begin{array}{cc}
												1, & qx = o_{A_d}\\
												0, & qx \neq o_{A_d}.
												\end{array}
												\right.
\]
Therefore, by Lebesgue's dominated convergence theorem we obtain
\begin{equation}\label{eq_1}
\left| \langle f, 1 \rangle \right|^2 = \sum_{qx = o_{A_d}} \mu_f(\{x\}).
\end{equation}
If we know in addition that $f \geq 0$, then by property (\ref{three}), the spectral measure 
$\mu_f$ is non-negative. Therefore, by use of equations (\ref{eq-zero}) and (\ref{eq_1}) we get that for all non-zero points $x \in A_d$ with $qx = o_{A_d}$ we have
\[
\mu_f(\{x\}) = 0.
\]
In particular, we have that $\mu_f(\{x_0\}) = 0$.
This finishes the proof of Proposition \ref{equid2}, if we choose $f = 1_{\tilde{B}}$.
\medskip

%

\end{document}